\documentclass[12pt]{amsart}
\usepackage{amssymb}  
\usepackage{latexsym} 
\usepackage[all]{xy}
\usepackage{comment}
\usepackage{url}

\newcommand{\Aff}{{\mathbb A}}

\newcommand{\C}{{\mathbb C}}

\newcommand{\eps}{{\varepsilon}}
\newcommand{\F}{{\mathbb F}}
\newcommand{\Gm}{{\mathbb G}_m}

\newcommand{\OO}{{\mathcal O}}
\newcommand{\I}{{\rm{I}}}
\newcommand{\isom}{\cong}

\newcommand{\la}{\lambda}
\newcommand{\Reg}{\operatorname{Reg}}

\newcommand{\PSL}{\operatorname{PSL}}
\newcommand{\NS}{\operatorname{NS}}
\newcommand{\PP}{{\mathbb P}}
\newcommand{\Q}{{\mathbb Q}}
\newcommand{\Qbar}{{\overline{\mathbb Q}}}
\newcommand{\Fbar}{{\overline{\mathbb F}}}

\newcommand{\rank}{{\operatorname{rank}}}

\newcommand{\Z}{{\mathbb Z}}

                       {\hspace*{\fill}\nobreak$\Box$\par\medskip}

\newtheorem{Proposition}{Proposition}[section]
\newtheorem{Theorem}[Proposition]{Theorem}
\newtheorem{Corollary}[Proposition]{Corollary}
\newtheorem{Lemma}[Proposition]{Lemma}

\theoremstyle{definition}

\newtheorem{Remark}[Proposition]{Remark}

\begin{document}
\date{26th April 2018}
\title[Congruences of elliptic curves]
{Explicit moduli spaces for \\ congruences of elliptic curves}

\author{Tom~Fisher}
\address{University of Cambridge,
          DPMMS, Centre for Mathematical Sciences,
          Wilberforce Road, Cambridge CB3 0WB, UK}
\email{T.A.Fisher@dpmms.cam.ac.uk}


\renewcommand{\baselinestretch}{1.1}
\renewcommand{\arraystretch}{1.3}

\renewcommand{\theenumi}{\roman{enumi}}

\begin{abstract}
  We determine explicit birational models over $\Q$ for the modular
  surfaces parametrising pairs of $N$-congruent elliptic curves in all
  cases where this surface is an elliptic surface.  In each case we
  also determine the rank of the Mordell-Weil lattice and the
  geometric Picard number.
\end{abstract}

\maketitle

\section{Introduction}

Let $N \ge 2$ be an integer.  A pair of elliptic curves are said to be
{\em $N$-congruent}, if their $N$-torsion subgroups are isomorphic as
Galois modules.  Such an isomorphism raises the Weil pairing to the
power $\eps$ for some $\eps \in (\Z/N\Z)^\times$. In this situation we
say that the $N$-congruence has {\em power $\eps$}. Since
multiplication-by-$m$ on one of the elliptic curves (for $m$ an
integer coprime to $N$) changes $\eps$ to $m^2 \eps$, we are only ever
interested in the class of $\eps \in (\Z/N\Z)^\times$ mod squares.

Let $Z(N,\eps)$ be the surface that parametrises pairs of
$N$-congruent elliptic curves with power $\eps$. This is a surface
defined over $\Q$. We only consider $Z(N,\eps)$ up to birational
equivalence. Kani and Schanz \cite[Theorem 4]{KS} classified the
geometry of these surfaces, explicitly determining the pairs
$(N,\eps)$ for which $Z(N,\eps)$ is birational over $\C$ to either (i)
a rational surface, (ii) an elliptic $K3$-surface, (iii) an elliptic
surface with Kodaira dimension one (also known as a properly elliptic
surface), or (iv) a surface of general type.  In case (i) it is known
that the surface is rational over $\Q$.  We show in cases (ii) and
(iii) that the surface is birational over $\Q$ to an elliptic surface,
determining in each case a Weierstrass equation for the
generic fibre as an elliptic curve over $\Q(T)$.  One application of
these explicit equations is that we are then able to use the methods
of van Luijk and Kloosterman to compute the geometric Picard number of
each surface.

The problem of computing $Z(N,\eps)$ is closely related to that of
computing the modular curves $X_E(N,\eps)$ parametrising the elliptic
curves $N$-congruent (with power $\eps$) to a given elliptic curve
$E$. Equations for $X_E(N,\eps)$, and the family of curves it
parametrises, have been determined as follows. The cases $N \le 5$
were treated by Rubin and Silverberg \cite{RubinSilverberg},
\cite{RubinSilverberg2}, \cite{Silverberg} for $\eps = 1$, and by
Fisher \cite{g1hess}, \cite{enqI} for $\eps \not= 1$. The case $N=7$
was treated by Halberstadt and Kraus \cite{HK} for $\eps = 1$, and by
Poonen, Schaefer and Stoll \cite{PSS} for $\eps \not= 1$. The case
$N=8$ was treated by Chen \cite{Chen8}, and the cases $N=9$ and $N=11$
by Fisher \cite{7and11}, \cite{congr9}.

If $N$ is not a prime power, then in principle we obtain equations for
$X_E(N,\eps)$ as a fibre product of modular curves of smaller level.
Equations that are substantially better than this have been obtained
in the case $(N,\eps) = (6,1)$ by Rubin and
Silverberg~\cite{RubinSilverberg6}, and independently Papadopoulos
\cite{Pap6}, and in the cases $(N,\eps) = (12,1)$ and $(12,7)$ by Chen
\cite[Chapter 7]{ChenThesis}.  Chen also gives equations for
$X_E(N,\eps)$ in the cases $(N,\eps) = (6,5)$ and $(10,1)$.

The equations for $X_E(N,\eps)$ do immediately give us equations for
$Z(N,\eps)$, but unfortunately this does not always make it easy to
find the elliptic fibrations.  The main purpose of this note is to
record the transformations that work in each case.

According to \cite[Theorem 4]{KS} the surface $Z(N,\eps)$ is rational
over $\C$ for all $N \le 5$, and in the cases $N = 6,7,8$ with
$\eps = 1$. In each of these cases $Z(N,\eps)$ is rational over $\Q$,
as follows (see \cite[Chapter 8]{ChenThesis}) from the results cited
above.

In our terminology, it is part of the definition of an elliptic
surface that it has a section. As we describe below, some of the cases
in the next two theorems were already treated in \cite{Chen8},
\cite{congr9}, \cite{K}.

\begin{Theorem}
\label{thm1}
The surfaces $Z(N,\eps)$ that are birational over $\C$ to an elliptic
$K3$-surface, are in fact birational over $\Q$ to an elliptic surface.
The generic fibres are the elliptic curves over $\Q(T)$ with the
following Weierstrass equations.
\begin{align*}
&Z(6,5): & 
&y^2 + 3 T (T - 2)xy  + 2 (T - 1) (T + 2)^2  (T^3 - 2)y
= x^3 
    -6 (T - 1) (T^3 - 2) x^2, \\
&Z(7,3): &
&y^2 = x^3 +  (4 T^4 + 4 T^3 -  51 T^2 - 2 T  - 50) x^2 
   + (6 T + 25) (52 T^2 - 4 T + 25) x, \\
&Z(8,3): &
&y^2 = x^3 - (3 T^2 - 7) x^2 - 4 T^2 (4 T^4 - 15) x 
    + 4 T^2 (53 T^4 + 81 T^2 + 162), \\
&Z(8,5): &
&y^2 = x^3 - 2(T^2 + 19) x^2 - (4 T^2 - 49) (T^4 - 6 T^2 + 25) x, \\
&Z(9,1): &
&y^2 +  (6 T^2 + 3 T + 2)xy + T^2 (T + 1) (4 T^3 + 9 T + 9) y
   \\ &&& \hspace{13em} = x^3 - (16 T^4 + 12 T^3 + 9 T^2 + 6 T + 1)x^2, \\
&Z(12,1): &
&y^2 + 2  (5 T^2 + 9) xy  + 96 (T^2 + 3) (T^2 + 1)^2 y 
   = x^3 + (T^2 + 3) (11 T^2 + 1)x^2.
\end{align*}
\end{Theorem}

\begin{Theorem}
\label{thm2}
The surfaces $Z(N,\eps)$ that are birational over $\C$ to a properly
elliptic surface, are in fact birational over $\Q$ to an elliptic
surface.  The generic fibres are the elliptic curves over $\Q(T)$ with
the following Weierstrass equations.
\begin{align*}
&Z(8,7): &
&y^2 = x^3 + 2(4 T^6 - 15 T^4 + 14 T^2 - 1) x^2 
+ (T^2 - 1)^4 (16 T^4 - 24 T^2 + 1) x, \\
&Z(9,2): & 
&y^2 + 3( 4 T^3 + T^2 - 2)xy + (T - 1)^3 (T^3 - 1) (4 T^3 - 3 T - 7)y
\\ &&& \hspace{7em} = x^3 -3 (T + 1) (T^3 - 1) (9 T^2 + 2 T + 1)x^2, 
\end{align*}
\begin{align*}
&Z(10,1): &
&y^2 -(3 T - 2) (6 T^2 - 5 T - 2)xy 
\\ &&& \hspace{3em}  -4 T^2 (T - 1)^2 (4 T^2 - 2 T - 1)
    (27 T^3 - 54 T^2 + 16 T + 12)y 
\\ &&& \hspace{7em}   = x^3 + T^2 (T - 1) (27 T^3 - 54 T^2 + 16 T + 12)x^2, \\
&Z(10,3): &
&y^2 + (T^3 - 8 T^2 - 9 T - 8)xy + 
    2 T^2 (T^3 - T^2 - 3 T - 3) (7 T^2 + 2 T + 3)y 
\\ &&& \hspace{7em}  = x^3 + 2 (3 T + 2) (T^3 - T^2 - 3 T - 3)x^2, \\
&Z(11,1): & 
&y^2 + (T^3 + T)xy = x^3 - 
      (4 T^5 - 17 T^4 + 30 T^3 - 18 T^2 + 4)x^2 
\\ &&& \hspace{14em}   
 + T^2 (2 T -1) (3 T^2 - 7 T + 5)^2 x.
\end{align*}
\end{Theorem}

Although we have not made it formally part of the statements of
Theorems~\ref{thm1} and~\ref{thm2}, our methods do also give the
moduli interpretations of these surfaces. In other words, given a
point on one of these surfaces (away from a certain finite set of
curves) we can determine the corresponding pair of $N$-congruent
elliptic curves. The fact that $N$-congruent elliptic curves over $\Q$
have traces of Frobenius (at all primes of good reduction) that are
congruent mod $N$, then provides some very useful check on our
calculations.

The second part of the following corollary was conjectured by Kani and
Schanz \cite[Conjecture 5]{KS}, and its proof (for $\eps = 1$) was
completed by Zexiang Chen in his PhD thesis \cite{ChenThesis}.  For
$N$ sufficiently large it is expected (with variants of this
conjecture variously attributed to Frey, Mazur, Kani and Darmon) that
the conclusions of the corollary are false.

\begin{Corollary}
  Let $N \le 12$ and $\eps \in (\Z/N\Z)^\times$ with $\eps =1$ if
  $N=11$ or $12$.
\begin{enumerate}
\item There are infinitely many pairs of non-isogenous elliptic curves
  over $\Q(T)$ that are $N$-congruent with power $\eps$.
\item There are infinitely many pairs of non-isogenous elliptic curves
  over $\Q$ that are $N$-congruent with power $\eps$.
\end{enumerate}
Moreover the $j$-invariants $j_1$ and $j_2$ of the elliptic curves in
(i) (resp. (ii)) correspond to infinitely many curves (resp. points)
in the $(j_1,j_2)$-plane.
\end{Corollary}

\begin{proof} In Table~\ref{tabMW} we list at least one $\Q$-rational
  section of infinite order for each of the elliptic surfaces in
  Theorems~\ref{thm1} and~\ref{thm2}.  This proves the first part.
  The second part follows by specialising $T$.  See the proof of
  \cite[Theorem~1.5]{congr9} for further details.  The final sentence
  of the statement is included to guard against various ``cheat''
  proofs, where new examples are generated from old by taking
  quadratic twists, or making substitutions for $T$.
\end{proof}

\begin{Remark} If elliptic curves $E_1$ and $E_2$ are $N$-congruent
  with power $\eps = -1$, then the quotient of $E_1 \times E_2$ by the
  graph of the $N$-congruence is a principally polarised abelian
  surface. The surface $Z(N,-1)$ may then be interpreted as a Hilbert
  modular surface, parametrising degree $N$ morphisms from a genus $2$
  curve to an elliptic curve.  At the outset of our work, this moduli
  interpretation had not been made explicit for any $N > 5$.
  Remarkably however, this approach has been used by A. Kumar \cite{K}
  to independently obtain results equivalent to the first two parts of
  Theorem~\ref{thm1} and the first three parts of
  Theorem~\ref{thm2}. As far as we are aware, his methods do not
  generalise to $\eps \not= -1$.
\end{Remark}

In Tables~\ref{tab1} and~\ref{tab2} we record some further data
concerning the elliptic surfaces in Theorems~\ref{thm1}
and~\ref{thm2}.  Since a K3-surface may admit many elliptic
fibrations, the data in Table~\ref{tab1} comes with the caveat that it
relates to the elliptic fibration we happened to find in
Theorem~\ref{thm1}. Since a properly elliptic surface has a unique
elliptic fibration, there is no such caveat for Table~\ref{tab2}.  We
list in each case the Kodaira symbols of the singular fibres (with
bracketing to indicate fibres that are Galois conjugates), the order
of the torsion subgroup over $\Q(T)$, the ranks of the group of sections over
$\Q(T)$ and $\Qbar(T)$, and finally the geometric Picard number
$\rho$. The lower bounds on the ranks are immediate from the
independent sections of infinite order listed in Tables~\ref{tabMW}
and~\ref{tabMWgeom}. The upper bounds on the ranks, and the geometric
Picard numbers are justified in Section~\ref{sec:pic}.

\begin{table}[ht]
\caption{The elliptic K3-surfaces in Theorem~\ref{thm1}}
\label{tab1}
$\begin{array}{llcccc}
(N,\eps) & \multicolumn{1}{c}{\text{singular fibres}} 
& |{\rm tors}| & \rank/\Q & \rank/\Qbar & \rho \\ \hline
(6,5) &  (\I_2, \I_2), \I_3,  (\I_3, \I_3, \I_3), \I_4, \I_4 & 1 & 2 & 2 & 20 \\ 
(7,3) & \I_1, \I_2,  (\I_2, \I_2),  (\I_2, \I_2), \I_3, \I_{10} & 2 & 2 & 2 & 20 \\ 
(8,3) &  (\I_1, \I_1), \I_2,  (\I_2, \I_2),  (\I_2, \I_2),  (\I_3, \I_3), \I_0^* & 1 & 4 & 5 & 20 \\ 
(8,5) & \I_2, \I_2,  (\I_2, \I_2),  (\I_2, \I_2),  (\I_3, \I_3), \I_0^* & 2 & 2 & 4 & 20 \\ 
(9,1) &  (\I_1, \I_1, \I_1), \I_2,  (\I_2, \I_2, \I_2), \I_3, \I_4, \I_0^* & 1 & 3 & 4 & 19 \\ 
(12,1) &  (\I_1, \I_1, \I_1, \I_1, \I_1, \I_1, \I_1, \I_1),  (\I_4, \I_4),  (\I_4, \I_4) & 1 & 3 & 5 & 19 
\end{array}$
\end{table}

\begin{table}[ht]
\caption{The properly elliptic surfaces in Theorem~\ref{thm2}}
\label{tab2}
$\begin{array}{llcccc}
(N,\eps) & \multicolumn{1}{c}{\text{singular fibres}} 
& |{\rm tors}| & \rank/\Q & \rank/\Qbar & \rho \\ \hline
(8,7) & \I_2,  (\I_2, \I_2),  (\I_2, \I_2),  (\I_3, \I_3), \I_4, \I_8, \I_8 & 2 & 1 & 2 & 30 \\ 
(9,2) &  (\I_2, \I_2, \I_2),  (\I_3, \I_3),  (\I_3, \I_3, \I_3), \I_9, \I_0^* & 1 & 2 & 2 & 29 \\ 
(10,1) &  (\I_2, \I_2),  (\I_2, \I_2),  (\I_3, \I_3, \I_3), \I_5, \I_{10}, {\rm{IV}} & 1 & 1 & 1 & 28 \\ 
(10,3) &  (\I_1, \I_1, \I_1), \I_2,  (\I_2, \I_2),  (\I_2, \I_2),  (\I_3, \I_3, \I_3), \I_4, \I_4, \I_6 & 1 & 3 & 4 & 28 \\ 
(11,1) &  (\I_1, \I_1, \I_1), \I_2,  (\I_2, \I_2, \I_2), \I_3, \I_4,  (\I_4, \I_4), \I_{10} & 2 & 2 & 2 & 28 
\end{array}$ 
\end{table}

\begin{table}[ht]
\caption{Mordell-Weil generators over $\Q(T)$}
\label{tabMW}
$\begin{array}{ll}
(N,\eps) & \multicolumn{1}{c}{\text{ $x$-coordinates of 
independent sections of infinite order}} \\ \hline
(6,5) &    0,\,\, 2 T^4 - 4 T, \\ 
(7,3) &    4 T^2 + 20 T + 25,\,\, 6 T + 25, \\
(8,3) &    -7,\,\, -T^2 + 9,\,\, -4 T^2 - 6 T,\,\, (4 T^5 - 2 T^4 + 10 T^3 + 6 T^2 + 18 T)/( T - 1)^2, \\
(8,5) &    -4 T^2 + 49,\,\, 2 T^3 + 19 T^2 + 60 T + 63, \\
(9,1) &    0,\,\, 4 T^4 + 2 T^3 - 2 T^2,\,\, 4 T^4 + 4 T^3 + 9 T^2 + 18 T + 9, \\ 
(12,1) &   0,\,\, -12 T^4 - 24 T^2 - 12,\,\, 4 T^6 + 12 T^4 - 4 T^2 - 12, \\
(8,7) &   4 T^6 + 4 T^5 - 9 T^4 - 10 T^3 + 4 T^2 + 6 T + 1, \\
(9,2) &   0,\,\, 2 T^5 - 8 T^3 + 4 T^2 + 6 T - 4, \\
(10,1) &   0, \\
(10,3) &   0,\,\, 2 T^5 - 4 T^4 - 4 T^3 + 6 T,\,\, 4 T^5 - 2 T^4 - 14 T^3 - 18 T^2 - 6 T, \\
(11,1) &   T^4 + 4 T^2 + 4,\,\, 3 T^2 - 7 T + 5.
\end{array}$ 
\end{table}

\begin{table}[ht]
\caption{Additional Mordell-Weil generators over $\overline{\Q}(T)$}
\label{tabMWgeom}
$\begin{array}{lcl}
(N,\eps) & d & \multicolumn{1}{c}{\text{ section of infinite order defined over $\Q(\sqrt{d})$}} \\ \hline
(8,3) & -2 & (-2 T^4 - 5 T^2 - 9, (2 T^6 + 5 T^4 + 20 T^2 + 9) \sqrt{-2}),  \\ 
(8,5) & -3 & (-2 T^3 + T^2 + 18 T - 35, (12 T^3 - 6 T^2 - 108 T + 210)\sqrt{-3}),  \\
(8,5) & -1 & (16 T^2 - 196,(8 T^4 - 346 T^2 + 3038)\sqrt{-1} ), \\
(9,1) & -3 & (-(19/3) T^4 - 15 T^3 - 9 T^2, \,\, \ldots \,\, ), \\
(12,1) & -3 & (-12 T^4 - 40 T^2 - 12, \,\, \ldots \,\, ), \\
(12,1) & -1 & (-16 T^4 - 64 T^2 - 48, \,\, \ldots \,\, ), \\
(8,7) & -3 & (-4 T^6 - 20 T^5 - 39 T^4 - 36 T^3 - 14 T^2 + 1,
\,\, \ldots \,\, ), \\
(10,3) & -3 & 
(-7 T^6 - 23 T^5 - 30 T^4 - 15 T^3 - 9 T^2, \,\, \ldots \,\, ).
\end{array}$ 
\end{table}

We organise the proofs of Theorems~\ref{thm1} and~\ref{thm2} as
follows.  The cases $N=8$ and $N=9$ were already treated in
\cite{Chen8}, \cite{congr9}, by starting from equations for
$X_E(N,\eps)$. In Section~\ref{sec:via} we use a similar approach to
treat the cases $(N,\eps) = (7,3)$, $(11,1)$ and $(12,1)$. Then in
Section~\ref{sec:S3} we treat the cases $(N,\eps) = (6,5)$, $(10,1)$
and $(10,3)$ by exhibiting $Z(N,\eps)$ as a degree $3$ cover of a
K3-surface.

The calculations described in this paper were carried out using Magma
\cite{Magma}. Accompanying Magma files 
are available from the author's website.  We assume that
the reader is familiar with the standard techniques for putting an
elliptic curve in Weierstrass form, as described in \cite[\S 8]{CaL},
or as implemented in Magma.

\section{Proofs via equations for $X_E(N,\eps)$}
\label{sec:via}

We prove Theorems~\ref{thm1} and~\ref{thm2} in the cases
$(N,\eps) = (7,3)$, $(11,1)$ and $(12,1)$.  The case
$(N,\eps) = (7,3)$ was treated in \cite[Section 8.2]{ChenThesis}, but
as this has not been published before, we include the details for
completeness.

\subsection*{Case ({\em N},\boldmath$\eps$) = (7,3)}
Let $E$ be the elliptic curve $y^2 = x^3 + ax + b$.  The following
equation for $X_E(7,3)$, as a quartic curve in $\PP^2$, was computed
by Poonen, Schaefer and Stoll \cite[Section 7.2]{PSS}, building on work of
Halberstadt and Kraus \cite{HK}.
\begin{align*}
F(a,b;x,y,z) = -a^2 & x^4 + 2 a b x^3 y - 12 b x^3 z - 6(a^3 + 6 b^2) x^2 y^2 
      + 6 a x^2 z^2 
\\ & + 2 a^2 b x y^3 - 12 a b x y^2 z + 18 b x y z^2 + (3 a^4 + 19 a b^2) y^4 
\\ & - 2(4 a^3 + 21 b^2) y^3 z + 6 a^2 y^2 z^2 - 8 a y z^3 + 3 z^4.
\end{align*}
Replacing $E$ by a quadratic twist does not change the isomorphism
class of $X_E(7,3)$. This is borne out by the identity
\[ F(\la^2 a,\la^3 b ; \la x,y, \la^2 z ) = \la^8 F(a,b;x,y,z). \]

The surface $Z(7,3)$ is the quotient of
$\{F = 0\} \subset \Aff^2 \times \PP^2$ by this $\Gm$-action.  We have
$F(y,x_1 y;x_2,T,y) = y^2 (c y^2 + h y - f)$ where
\begin{align*}
c &= (T^2 + 1) (3 T^2 - 8 T + 3), \\
h &= T^3 (19 T - 42) x_1^2 + 2 T (T - 3)^2 x_1 x_2 - 6 (T^2 - 1)x_2^2, \\
f &= 36 T^2 x_1^2x_2^2 - 2(T - 6) x_1 x_2^3 + x_2^4.
\end{align*}
Therefore $Z(7,3)$ is birational to the total space for the genus one
curve over $\Q(T)$ with equation $Y^2 = h^2 + 4 cf$. This is a double
cover of $\PP^1$ with a rational point above $(x_1:x_2) = (1:0)$.
Putting this elliptic curve in Weierstrass form, and replacing $T$ by
$(6T - 3)/(4T + 4)$, gives the equation in the statement of
Theorem~\ref{thm1}.

\subsection*{Case ({\em N},\boldmath$\eps$) = (11,1)}
Let $E$ be the elliptic curve $y^2 = x^3 + ax + b$.  Equations for
$X_E(11,1)$ as a curve in $\PP^4$ were computed in
\cite[Theorem~1.2]{7and11}. These equations are the $4 \times 4$
minors of the $5 \times 5$ Hessian matrix of the cubic form
\begin{align*}
F(a,b;v,w,&x,y,z) = v^3 + a v^2 z - 2 a v x^2 + 4 a v x y - 6 b v x z + a v y^2 
    + 6 b v y z \\ & + a^2 v z^2 - w^3 + a w^2 z - 4 a w x^2 - 12 b w x z 
    + a^2 w z^2 - 2 b x^3 + 3 b x^2 y \\ & + 2 a^2 x^2 z + 6 b x y^2 
    + 4 a b x z^2 + b y^3 - a^2 y^2 z + a b y z^2 + 2 b^2 z^3.
\end{align*}
Replacing $E$ by a quadratic twist does not change the isomorphism
class of $X_E(11,1)$. This is borne out by the identity
\[ F(\la^2 a,\la^3 b;\la^2 v,\la^2 w,\la x,\la y,z) = \la^6
F(a,b;v,w,x,y,z). \]
We may describe $Z(11,1)$ as the quotient of a $3$-fold in
$\Aff^2 \times \PP^4$ by this $\Gm$-action.

We start by using the discriminant condition $4a^3 + 27b^2 \not=0$
to simplify the equations for $X_E(11,1)$. The polynomials
\begin{align*}
F_1 &= v z + 2 w z + x^2 - x y - y^2, \\
F_2 &=     a x z + b z^2 - v x + v y - 2 w x, \\
F_3 &=     a^2 z^2 + 2 a w z - 4 a x^2 - 12 b x z - 3 w^2, \\ 
F_4 &=     a^2 z^2 + 2 a v z - 2 a x^2 + 4 a x y + a y^2
        - 6 b x z + 6 b y z + 3 v^2, \\
F_5 &=     2 a^2 y z - a b z^2 - 4 a v x - 2 a v y - 6 b v z
        - 3 b x^2 - 12 b x y - 3 b y^2,
\end{align*}
are linear combinations of the derivatives of $F$, where the matrix
implicit in taking these linear combinations is invertible if
$4a^3 + 27b^2 \not=0$. Now $X_E(11,1)$ is defined by the $4 \times 4$
minors of the $5 \times 5$ Jacobian matrix ($M$ say) of
$F_1, \ldots, F_5$.

We make the substitutions
\begin{align*}
a &= (4 U + 3 x_3) x_4 - 3 x_5^2, \\
b &= x_2 (x_1 + x_3) x_4 - (4 U + 3 x_3) x_4 x_5 + 2 x_5^3, \\
(v,w,x,y,z) &= ( x_2 x_4 + x_4 x_5 + x_5^2, x_3 x_4 - x_5^2, x_5, x_4, 1 ). 
\end{align*}
We have $4a^3 + 27b^2 = x_4 h$ for some polynomial $h$.  We add $x_5$
times the first row of $M$ to the second row. We then divide all but
the first row by $x_4$. Let $I \subset \Q[U,x_1,x_2,x_3,x_4,x_5]$ be
the ideal generated by the $4 \times 4$ minors of $M$, and
\[ J = \{f \in \Q[U,x_1,x_2,x_3,x_4,x_5] : x_2 h f \in I \}. \]
Using the Gr\"obner basis machinery in Magma we find that 
$J \cap \Q[U,x_1,x_2,x_3,x_4]$ is generated by $3$ 
homogeneous polynomials of degree $4$. These define a surface
in $\PP^4$ of degree $12$. 
By the substitution 
\[ T = \frac{4 (2 x_1 - x_2 + x_3) U + (x_1 x_2 - x_2^2 + x_3 x_4)}
             {2 (2 x_4 - x_2) U + 2 (x_1 x_2 - x_2^2 + x_3 x_4)} \]
this surface is birational to the 
surface $\{Q_1 = Q_2 = 0\}$ in $\Aff^1 \times \PP^3$ where 
\begin{align*}
Q_1 &= 4 (T - 2) x_1 x_2 + 8 x_1 x_3 + (T - 2)^2 x_2^2
  \\ & \hspace{7em}  + 4 (T - 2) x_2 x_3 + 2 (T^2 - T + 1) x_2 x_4
   + 4 x_3^2 - 4 T x_3 x_4, \\ 
Q_2 &= 8 x_1^2 + 16 x_1 x_3 - 4 (2 T - 1) x_1 x_4 
                        - T (T - 2) x_2^2 - T^2 x_2 x_3
 \\ &  \hspace{7em} - 2 (T^2 - T + 1) x_2 x_4
                   - 2 (T - 4) x_3^2 - 2 (T^2 + 4 T - 1) x_3 x_4.
\end{align*}
These same equations define a genus one curve in $\PP^3$ defined over
$\Q(T)$, with a rational point at $(x_1: x_2: x_3: x_4) =
(0:0:0:1)$.
Putting this elliptic curve in Weierstrass form gives the equation in
the statement of Theorem~\ref{thm2}.

\subsection*{Case ({\em N},\boldmath$\eps$) = (12,1)}
Let $E$ be the elliptic curve $y^2 = x^3 + ax + b$.
Equations for $X_E(12,1)$ as a curve in $\PP^5$ were computed in
\cite[Theorem 1.7.10]{ChenThesis}. These equations 
are $F_0 = F_1 = F_2 = F_3 = 0$ where
\begin{align*}
 F_0 &= -X^2 Z + a X Y^2 + 6 b Y^3 - 6 a Y^2 Z - 12 Z^3, \\
 F_1 &= X^2 + 12 X Z + 36 Z^2 - 2 u_0 u_2 - u_1^2 + a u_2^2, \\
 F_2 &= 4 a X Y + 36 b Y^2 - 24 a Y Z - 2 u_0 u_1 + 2 a u_1 u_2 + b u_2^2, \\
 F_3 &=  8 a X Z - 4 a^2 Y^2 - u_0^2 + 2 b u_1 u_2. 
\end{align*}
These polynomials satisfy
\[ F_i(\la^2 a,\la^3 b;\la X,Y,\la Z,\la^2 u_0,\la u_1,u_2) =
\la^{m_i} F_i(a,b;X,Y,Z,u_0,u_1,u_2) \]
where $(m_0,m_1,m_2,m_3) = (3,2,3,4)$. Again, it is our aim to
quotient out by this $\Gm$-action. We do this by setting
$(X + 6 Z)Y = u_2^2$.  Specifically, we substitute
$(X,Y,Z,u_0,u_1,u_2) = (x^2 - 6 y,1,y,v x,w x,x)$ and then solve for
$a$ and $b$ so that the first two equations are satisfied.  In the
remaining two equations we substitute
\[ v = 2 (w - 2 y) y + \frac{T + 1}{T - 1} (x^2 (y + 1) - (w + 2 y)^2). \]
The resultant with respect to $w$ is $f(T) x^{14} y^{2} g(x,y)^2 
h(x,\widetilde{y})$ where $f(T)$ is a rational function in $T$, 
$g(x,y) = x^6 (y + 1) - 9 y^2 (x^2 + 4)^2$,
\begin{align*} 
h(x,y) = (T + 1)^2 & x^2 y^2 + (T + 2) (T^2 + 3)^2 x^2 \\
   & - 4 (T - 1) (T + 3)^2 x y + 4 (T + 3)^2 y^2 + 12 T (T + 1)^2 (T + 3)^2
\end{align*}
and $\widetilde{y} = (144 (T + 1) y + (T + 3)^2 ((T - 3) x^2 + 12 (T + 1)))/(8 (T + 3) x)$. 
Therefore $Z(12,1)$ is birational to the total space
for the genus one curve $C = \{h = 0\}$ in $\Aff^2$ defined 
over $\Q(T)$. Replacing $x$ by $2(T + 3)/(T^2 + 3)x$, and 
completing the square in $y$ shows that $C$ has equation
\begin{equation}
\label{g1fib}
 Y^2 = -(T + 2) x^4 - (4 T^3 + 5 T^2 + 6 T + 9) x^2 - 3 T (T^2 + 3)^2. 
\end{equation}
This gives a genus one fibration on $Z(12,1)$ defined over $\Q$, but
without a $\Q$-rational section. Indeed the fibres with $T > 0$ have
no real points.

We now find another genus one fibration that does have a $\Q$-rational
section.  Let $F(x_1,x_2,x_3)$ be the unique homogeneous polynomial of
degree $6$ with the property that $F(x,T,1)$ is the right hand side
of~\eqref{g1fib}.  Then $F$ is the discriminant of the 
following quadratic in $T$.
\[ x_1^2 x_2 + (T^2 + 2) x_1^2 x_3 + 2 T x_1 x_2^2 - 2 T x_1 x_2 x_3 +
T^2 x_2^3 + 3 x_2^2 x_3 + 3 T^2 x_2 x_3^2 + 9 x_3^3=0 \]
This same equation defines a genus one curve in $\PP^2$ defined over
$\Q(T)$, with a rational point at $(x_1: x_2: x_3) = (1:0:0)$. Putting
this elliptic curve in Weierstrass form gives the equation in the
statement of Theorem~\ref{thm1}.

\section{Degree $3$ covers of K3-surfaces}
\label{sec:S3}

We prove Theorems~\ref{thm1} and~\ref{thm2} in the cases
$(N,\eps) = (6,5)$, $(10,1)$ and $(10,3)$.  In the first of these
cases, Chen's equations for $X_E(6,5)$ already give a genus one
fibration on $Z(6,5)$, but one without a section.  The content of
Theorem~\ref{thm1} in this case is that we can find another genus one
fibration that does have a section.

For $N$ an odd integer, let $Z^*(N,\eps)$ be the double cover of
$Z(N,\eps)$ that param\-etrises pairs of elliptic curves whose ratio of
discriminants is a square.

\begin{Theorem}
\label{thm:Zstar}
If $(N,\eps) = (3,2)$, $(5,1)$ or $(5,2)$ then $Z^*(N,\eps)$ is a double
cover of $\PP^2$, ramified over the union of two cuspidal cubics, with 
equation 
\begin{equation} 
\label{dc}
y^2 = F_{+}(u,v,w) F_{-}(u,v,w) 
\end{equation} where 
\begin{align*}
&Z^*(3,2): & F_{\pm} &= u (u + 3 v \pm w)^2 + 4 v^3, \\
&Z^*(5,1): & F_{\pm} &= u(u^2 - 11 u v - v^2)+ w^2 (12 u + v)  
                      \pm 2 w (3 u^2 - 4 u v + 4 w^2), \\
&Z^*(5,2): & F_{\pm} &= u^2 (11 v + 8 w) + w^2 (8 u - v + 4 w) 
            \pm 2 u (2 v - w)(4u - v + 4w).
\end{align*}
\end{Theorem}
\begin{proof}
Let $E$ be the elliptic curve $y^2 = x^3 + ax + b$. We put 
$\Delta =  -4 a^3 - 27 b^2$, and define polynomials
\begin{align*}
f(x) &= x^3 + a x + b, \\
g(x) &= 3 a x^4 + 18 b x^3 - 6 a^2 x^2 - 6 a b x - a^3 - 9 b^2, \\
h(x) &= 3 a x^2 + 9 b x - a^2, \\
j(x) &= 27 b x^3 - 18 a^2 x^2 - 27 a b x - 2 a^3 - 27 b^2.
\end{align*}
If we assign the variables $x, a, b$ weights 1, 2, 3, then each of
these polynomials is homogeneous.  We note that
$j^2 = -4 h^3 - 27 \Delta f^2$.

\smallskip

\noindent
{\bf{Case ({\em N},\boldmath$\eps$) = (3,2)}.}
The following equations for the family of curves parametrised by
$X_E(3,2)$ are taken from~\cite[Section 13]{g1hess}.  Starting from the
Klein form\footnote{We obtain $D$ from ${\mathfrak D}(\xi,\eta)$ in
  \cite[Section 9]{g1hess} by putting $c_4 = -48a$, $c_6=-864b$,
  multiplying $\xi$ by $12$, and dividing through by $2^{12}3^{3}$.}
\[ D(\xi,\eta) = -27 a \xi^4 - 54 b \xi^3 \eta - 18 a^2 \xi^2 \eta^2 -
54 a b \xi \eta^3 + (a^3 - 27 b^2) \eta^4, \] we define
\[ A(\xi,\eta) = \frac{1}{108} \left| \begin{array}{cc} 
D_{\xi \xi} & D_{\xi \eta} \\ D_{\eta \xi} & D_{\eta \eta} \end{array} \right|, 
\quad \text{ and } \quad
B(\xi,\eta) = \frac{1}{36} \left| \begin{array}{cc} 
D_{\xi} & D_{\eta} \\ A_{\xi} & A_{\eta} \end{array} \right|, 
\]
where the subscripts denote partial derivatives. These forms satisfy
the syzygy
\begin{equation}
\label{klein:32}
-4 A^3 - 27 B^2 = 16 (4 a^3 + 27 b^2)^2 D^3. 
\end{equation}
The family of elliptic curves $3$-congruent to $E$ with power
$\eps = 2$ is given by
\[ y^2 = x^3 + A(\xi,\eta) x + B(\xi,\eta). \]

We dehomogenise by putting $(\xi,\eta)=(x,1)$. Then
$D = f_x^3 - 27 f^2 = j - 3 f_x h$ where $f_x = 3x^2 + a$.  The
quantities $(u,v,r,s) = (D, f_x h, h^3,3^6 \Delta f^4)$ are related by
\begin{equation}
\label{Z32eqn}
 (4 r + (u + 3 v)^2) (r u - v^3) = r s. 
\end{equation}
As we verify in Remark~\ref{rem:inverse} below, this is an equation
for $Z(3,2)$ in $\PP(1,1,2,3)$ where the coordinates $u,v,r,s$ have
weights $1,1,2,3$.  We see by~\eqref{klein:32} that, up to squared
factors, the ratio of discriminants is $s/u$. We substitute
$s = u w^2$ in~\eqref{Z32eqn} to give a quadratic in $r$ whose
discriminant is the polynomial $F_{+} F_{-}$ in the statement of the
theorem.

\medskip

\noindent
{\bf{Case ({\em N},\boldmath$\eps$) = (5,1)}.}
The following equations for the family of curves parametrised by
$X_E(5,1)$ are taken from~\cite[Section 13]{g1hess}.  Starting from the
Klein form\footnote{We obtain $D$ from ${\mathbf D}(\la,\mu)$ in
  \cite[Section 8]{g1hess} by putting $c_4 = -48a$, $c_6=-864b$,
  multiplying $\la$ by $12$, and dividing through by $2^{24}3^{12}$.}
\begin{align*}
D(\la,\mu) = \la&^{12} + 22 a \la^{10} \mu^2 + 220 b \la^9 \mu^3 
  - 165 a^2 \la^8 \mu^4 - 528 a b \la^7 \mu^5
  \\ & - 220 (a^3 + 12 b^2) \la^6 \mu^6 + 264 a^2 b \la^5 \mu^7
  - 165 a (5 a^3 + 32 b^2) \la^4 \mu^8 
  \\ & - 880 b (3 a^3 + 20 b^2) \la^3 \mu^9
  + 22 a^2 (25 a^3 + 168 b^2) \la^2 \mu^{10} 
  \\ & + 20 (19 a^4 b + 128 a b^3) \la \mu^{11}
  + (125 a^6 + 1792 a^3 b^2 + 6400 b^4) \mu^{12},
\end{align*}
we define
\[ A(\la,\mu) = \frac{1}{5808} \left| \begin{array}{cc} 
D_{\la \la} & D_{\la \mu} \\ D_{\mu \la} & D_{\mu \mu} \end{array} \right|,
\quad \text{ and } \quad
B(\la,\mu) = \frac{1}{360} \left| \begin{array}{cc} 
D_{\la} & D_{\mu} \\ A_{\la} & A_{\mu} \end{array} \right|,
\]
where the subscripts denote partial derivatives. These forms
satisfy the syzygy
\begin{equation}
\label{klein:51}
4 A^3 + 27 B^2 = (4 a^3 + 27 b^2) D^5. 
\end{equation}
The family of elliptic curves $5$-congruent to $E$ with power $\eps = 1$
is given by
\[ y^2 = x^3 + A(\la,\mu) x + B(\la,\mu). \]

We dehomogenise by putting $(\la,\mu)=(x,1)$. Then
\[ D = 4 kf - 3(f^2 + g)^2 + 32\Delta(f^2 + g),\]
where
$k(x) = f^3 + f j + 4 \Delta f + 3 g(x f_x - 2 f) = x^9 + 12 a x^7 +
84 b x^6 + \ldots$

The quantities
$(t,u,v,r,s) = (4 f,2 (f^2 + g),16 \Delta,4 k,D)$
are related by
\begin{equation}
\label{Z51eqns}
\begin{aligned}
 r^2 + s t^2 &= u (u^2 - 11 u v - v^2) + (12 u + v) s, \\
 r t &= 3 u^2 - 4 u v + 4 s.
\end{aligned}
\end{equation}
These are equations for $Z(5,1)$ in $\PP(1,2,2,3,4)$ where the
coordinates $t,u,v,r,s$ have weights $1,2,2,3,4$.  We see
by~\eqref{klein:51} that, up to squared factors, the ratio of
discriminants is $s$.  Putting $s= w^2$ we obtain from~\eqref{Z51eqns}
the equation
\[(r^2 - s t^2)^2 = (r^2 + s t^2)^2 - 4 s(r t)^2 = F_{+}(u,v,w) F_{-}(u,v,w)\]
where $F_{\pm}$ are the polynomials in the statement of the theorem.

\medskip

\noindent
{\bf{Case ({\em N},\boldmath$\eps$) = (5,2)}.}
The following equations for the family of curves parametrised by
$X_E(5,2)$ are taken from~\cite[Theorem 5.8]{enqI}.  Starting from the
Klein form\footnote{We obtain $D$ from ${\mathbf D}(\la,\mu)$ in
  \cite[Section 5]{enqI} by putting $c_4 = -48a$, $c_6=-864b$,
  multiplying $\la$ by $12$, and dividing through by $2^{36}3^{15}$.}
\begin{align*}
D(&\la,\mu) = (125 a^3 - 432 b^2) \la^{12}
    + 2430 a^2 b \la^{11} \mu
    - 22 a (25 a^3 - 378 b^2) \la^{10} \mu^2
    \\ & - 110 b (11 a^3 - 108 b^2) \la^9 \mu^3
    - 165 a^2 (5 a^3 - 27 b^2) \la^8 \mu^4
    - 132 a b (53 a^3 - 189 b^2) \la^7 \mu^5
    \\ & + 220 (a^6 - 123 a^3 b^2 + 81 b^4) \la^6 \mu^6
    + 132 a^2 b (19 a^3 - 297 b^2) \la^5 \mu^7
    \\ & - 165 (a^7 - 26 a^4 b^2 + 189 a b^4) \la^4 \mu^8
    - 110 (3 a^6 b - 34 a^3 b^3 + 135 b^5) \la^3 \mu^9
    \\ & - 22 a^2 (a^3 - 3 b^2) (a^3 + 27 b^2) \la^2 \mu^{10}
    - 10 a b (5 a^6 + 82 a^3 b^2 + 189 b^4) \la \mu^{11}
    \\ & + (a^9 - a^6 b^2 - 181 a^3 b^4 - 675 b^6) \mu^{12},
\end{align*}
we define
\[ A(\la,\mu) = \frac{1}{1452} \left| \begin{array}{cc} 
D_{\la \la} & D_{\la \mu} \\ D_{\mu \la} & D_{\mu \mu} \end{array} \right|, 
\quad \text{ and } \quad
B(\la,\mu) = \frac{-1\,\,}{180} \left| \begin{array}{cc} 
D_{\la} & D_{\mu} \\ A_{\la} & A_{\mu} \end{array} \right|,
\]
where the subscripts denote partial derivatives. These forms
satisfy the syzygy
\begin{equation}
\label{klein:52}
-4 A^3 - 27 B^2 = 16(4 a^3 + 27 b^2)^2 D^5. 
\end{equation}
The family of elliptic curves $5$-congruent to $E$ with power
$\eps = 2$ is given by
\[ y^2 = x^3 + A(\la,\mu) x + B(\la,\mu). \]

We dehomogenise by putting $(\la,\mu)=(x,1)$. Then 
\begin{equation}
\label{niceD}
D = 16 \Delta f^4 - g^3 + 4 (2 g^3 - g^2 j - 4 \Delta f^2 g).
\end{equation}
The quantities
$(r,s,v,w) = (\Delta f^4,\Delta f^2 g,g^3,2 g^3 - g^2 j - 4 \Delta f^2 g)$
are related by 
\begin{equation}
\label{Z52eqn}
r (4 s - 2 v + w)^2 + 27 r s v + s w^2 - s^2 (v - 4 w) = 0. 
\end{equation}
This is an equation for $Z(5,2)$ as a cubic surface in $\PP^3$.  We
see from~\eqref{klein:52} and~\eqref{niceD} that, up to squared
factors, the ratio of discriminants is $r (16 r - v + 4 w)$.  Putting
$r (16 r - v + 4 w) = (4 r - u)^2$, where $u$ is a new variable, and
using this equation to eliminate $r$ from~\eqref{Z52eqn}, we obtain a
quadratic in $s$ whose discriminant is the polynomial $F_{+}F_{-}$ in
the statement of the theorem.
\end{proof}

\begin{Remark}
\label{rem:inverse}
Let $(N,\eps) = (3,2)$, $(5,1)$ or $(5,2)$. We saw in the proof of
Theorem~\ref{thm:Zstar} that one model for $Z(N,\eps)$ is the weighted
projective plane $\PP(1,2,3)$ where the co-ordinates $x,a,b$ have
weights $1,2,3$. We mapped this to another model for $Z(N,\eps)$
defined by~\eqref{Z32eqn}, \eqref{Z51eqns} or~\eqref{Z52eqn}. The
inverse maps are as follows.
\begin{align*}
&(3,2) && \left\{
\begin{aligned}
x &= r + v^2, \\
a &= -3 r (r + u v + 2 v^2), \\
b &= r (u + 3 v) (r u + v^3) + 2 r^2 (r + 3 v^2),
\end{aligned} \right. \\
&(5,1) && \left\{
\begin{aligned} 
x & =  32 u - v + 5 t^2, \\
a & = -3 (8 u - v) (32 u - v) - 288 r t + 30 (28 u + v) t^2 - 75 t^4, \\
b & = -2 (32 u - v)^2 (4 u + v) - 144 (32 u - v) rt
+ 6 (32 u - v) (88 u - 5 v) t^2
\\ & \qquad \qquad + 1008 r t^3 - 150 (28 u + v) t^4 + 250 t^6, 
\end{aligned} \right. \\
&(5,2) && \left\{
\begin{aligned} 
x &= 4 r s + 4 r v + r w - s^2, \\
a &= 3 (8 r s^3 + 4 r s^2 v + 6 r s^2 w + r s w^2 - s^4), \\
b &= r^2 s (16 s^3 - 8 s^2 v - 24 s^2 w - 40 s v w - 15 s w^2 + 4 v w^2 - 2 w^3)  \\ & \qquad \qquad + r s^3 (24 s^2 + 8 s v + 34 s w + 7 w^2) - 2 s^6.
\end{aligned} \right.
\end{align*}
\end{Remark}

\begin{Remark}
  There are two naturally defined involutions on the K3-surfaces in
  Theorem~\ref{thm:Zstar}. The first switches the sign of $y$, and
  corresponds to swapping over the pair of $N$-congruent elliptic
  curves.  The second is given on $Z^*(3,1)$ and $Z^*(5,1)$ by
  switching the sign of $w$, and on $Z^*(5,2)$ by
  $(u,v,w,y) \mapsto (\widetilde{u},v,w,(\widetilde{u}/u)^2y)$ where
  $\widetilde{u} = u(v - 4w)/(8u - (v-4w))$. This second involution
  switches the choice of square root for the ratio of
  discriminants. The two involutions commute, and the second swaps
  over the curves $F_+=0$ and $F_-=0$.
\end{Remark}

\begin{Remark}
For a suitable parametrisation of the cuspidal cubic $F_+=0$, 
we obtain a family of elliptic curves with $j$-invariant
\begin{align*}
& (3,1): & j &= 27 (T - 3)^3 (T + 1)^3/T^3, \\
& (5,1): & j &= (T + 5)^3 (T^2 - 5)^3 (T^2 + 5 T + 10)^3/(T^2 + 5 T + 5)^5, \\
& (5,2): & j &= 125 T (2 T + 1)^3 (2 T^2 + 7 T + 8)^3/(T^2 + T - 1)^5.
\end{align*}
These correspond to $X_{\rm s}^+(3)$, $X_{\rm s}^+(5)$ and
$X_{\rm ns}^+(5)$, where $X_{\rm s}^+(N)$ and $X_{\rm ns}^+(N)$ are
the modular curves associated to the normaliser of a split or
non-split Cartan subgroup of level $N$. We may compute
$X_{\rm s}^+(N)$ as the quotient of $X_0(N^2)$ by the Fricke
involution, whereas the formula for $X_{\rm ns}^+(5)$ is taken from
\cite[Corollary 5.3]{IChen}. The use of these modular curves to
construct pairs of $N$-congruent elliptic curves is described further
in \cite{Halberstadt}.
\end{Remark}

Let $N$ be an odd integer and let $\eps \in (\Z/2N\Z)^\times$. Then
$X_E(2N,\eps) \to X_E(N,\eps)$ is geometrically a Galois covering with
Galois group $\PSL_2(\Z/2\Z) \isom S_3$.  Since elliptic curves which
are $2$-congruent have ratio of discriminants a square, it follows
that $Z(2N,\eps) \to Z^*(N,\eps)$ is a degree 3 cover. In the cases
$(2N,\eps)= (6,5)$, $(10,1)$ and $(10,3)$ the surface $Z(2N,\eps)$ has
an elliptic fibration. The pushfoward of a fibre gives a divisor class
$D$ on the K3-surface $Z^*(N,\eps)$ with $D^2 = 2$.  Using this
divisor class $D$ we may write $Z^*(N,\eps)$ as a double cover of
$\PP^2$. We have arranged (with the benefit of hindsight) that the
equations in Theorem~\ref{thm:Zstar} write $Z^*(N,\eps)$ as a double
cover of $\PP^2$ in exactly this way.

The equations for $Z(2N,\eps)$ in Theorems~\ref{thm1} and~\ref{thm2}
may be obtained from the equations for $Z^*(N,\eps)$ in
Theorem~\ref{thm:Zstar} as follows.  The tangent line to a general
point on the cuspidal cubic $F_{+}(u,v,w) = 0$ has equation:
\begin{align}
(2N,\eps) &= (6,5) &   (T^3 - 1) u + 3 (T - 1) v - w &=0, \\
(2N,\eps) &= (10,1) &  (T - 2)u  - T (T - 1)^2 v + 2 (T - 1) w &=0, \\
\label{tgt}
(2N,\eps) &= (10,3) &   T^3 u - (T + 1)v - T^2 w &= 0. 
\end{align}
We parametrise this line, and substitute into the right hand side of
the equation $y^2 = F_{+}(u,v,w) F_{-}(u,v,w)$.  After cancelling a
squared factor (which arises since we chose a tangent line) the right
hand side is a binary quartic with a linear factor.  We now have the
equation for a genus one curve over $\Q(T)$ with a rational point.
Putting this elliptic curve in Weierstrass form gives the equations
for $Z(6,5)$, $Z(10,1)$ and $Z(10,3)$ in Theorems~\ref{thm1}
and~\ref{thm2}.

It remains to show that these degree $3$ covers of $Z^*(N,\eps)$ are
the ones we wanted. We use the following lemma.

\begin{Lemma}
\label{lem:congr2}
Let $K$ be a field of characteristic not $2$ or $3$.  Elliptic curves
$E_1$ and $E_2$ over $K$ with $j$-invariants $j_1$ and $j_2$, with
$j_1,j_2 \not\in \{0,1728\}$, are $2$-congruent if and only if there
exist $m,x \in K$ satisfying $(j_1 - 1728)(j_2 - 1728) = m^2$ and
\[ x^3 - 3 j_1 j_2 x - 2 j_1 j_2 (m + 1728) = 0.\]
\end{Lemma}
\begin{proof} This follows from the formulae in \cite{RubinSilverberg2} 
or \cite[Sections 8 and 13]{g1hess} by a generic calculation.
\end{proof}

We illustrate the use of Lemma~\ref{lem:congr2} in the case
$(2N,\eps)=(10,3)$, the other cases being similar. Above each point
$(u:v:w) \in \PP^2$ there are a pair of points on $Z^*(5,2)$ possibly
defined over a quadratic extension. These points correspond to a pair
of elliptic curves, say with $j$-invariants $j_1$ and $j_2$.  A
calculation using the formulae in Remark~\ref{rem:inverse} shows that,
for $m$ a suitable choice of square root of
$(j_1 - 1728)(j_2 - 1728)$, we have
\begin{align*}
j_1 j_2 &= G_6(u,v,w) H(u,v,w)^2 \\
j_1 j_2 (m + 1728) &= G_9(u,v,w) H(u,v,w)^3
\end{align*}
where
\begin{align*}
G_6(u,v,w) &= 640 u^4 v^2 - 768 u^4 v w - 72 u^3 v^3 - 240 u^3 v^2 w 
    + \ldots \\
G_9(u,v,w) &= 6912 u^7 v^2 - 1376 u^6 v^3 - 14976 u^6 v^2 w 
    + \ldots  
\end{align*}
are irreducible homogeneous polynomials of degrees $6$ and $9$, and
$H \in \Q(u,v,w)$ is a rational function.  Finally we claim that the
polynomials
\begin{equation}
\label{cubic1}
X^3 - 3 G_6(u,v,w) X - 2 G_9(u,v,w) = 0,
\end{equation}
arising from Lemma~\ref{lem:congr2}, and 
\begin{equation}
\label{cubic2}
u T^3 - w T^2 - v T - v = 0,
\end{equation}
appearing in~\eqref{tgt}, define the same cubic extension. Indeed we
find by computer algebra that if~\eqref{cubic2} has root $T_0$
then~\eqref{cubic1} has root
\begin{align*} X_0 &= 3 u^2 (8 u - 3 v - 4 w) T_0^2
               + 12 u (2 u v - 4 u w + v w) T_0 \\
  & \qquad \qquad \qquad \qquad \qquad \qquad \qquad 
- 16 u^2 v + 6 u v^2 + 8 u w^2 - v w^2 + 4 w^3. 
\end{align*}

\section{Computing the Picard numbers}
\label{sec:pic}

Let $E/\Q(T)$ be one of the elliptic curves in Theorems~\ref{thm1}
and~\ref{thm2}.  We write $X \to \PP^1$ for the minimal fibred surface
with generic fibre $E$. The reduction of $E$ mod $p$ is an elliptic
curve $E_p/\F_p(T)$, and the reduction of $X$ mod $p$ is a surface
$X_p/\F_p$. We will always take $p$ to be a prime of good reduction.

Let $\overline{X} = X \times_\Q \Qbar$ and
$\overline{X}_p = X_p \times_{\F_p} \Fbar_p$.  The Shioda-Tate formula
\cite[Corollary~5.3]{Shioda} tells us that
\begin{equation}
\label{ST0}
\rank E(\Qbar(T)) + 2 + \textstyle\sum_{t \in \PP^1(\Qbar)} (m_t - 1) = \rank \NS(\overline{X}), 
\end{equation}
and
\begin{equation}
\label{STp}
\rank E_p(\Fbar_p(T)) + 2 + \textstyle\sum_{t \in \PP^1(\Fbar_p)} (m_t - 1) = \rank \NS(\overline{X}_p),
\end{equation}
where $m_t$ is the number of irreducible components in the fibre above
$t$.  We write $\rho$ and $\rho_p$ for the numbers on the right
of~\eqref{ST0} and~\eqref{STp}. These are the geometric Picard numbers
of $X$ and $X_p$.  The sections exhibited in Tables~\ref{tabMW}
and~\ref{tabMWgeom} give a lower bound for $\rank E(\Qbar(T))$ and
hence by~\eqref{ST0} a lower bound for $\rho$.  These lower bounds are
exactly the values recorded in Tables~\ref{tab1} and~\ref{tab2}.

Let $X \to \PP^1$ be a minimal elliptic surface with non-constant
$j$-invariant, and let $m = \chi(\OO_X)$. This may be computed from
the fact that sum of the Euler numbers of the singular fibres is
$12m$.  By \cite[Lemma IV.1.1]{Miranda} the Hodge diamond of $X$ is
\[ \begin{array}{ccccc}
& & h^{0,0} \\
& h^{1,0} && h^{0,1} \\
h^{2,0} && h^{1,1} && h^{0,2} \\
& h^{2,1} && h^{1,2} \\
&& h^{2,2}
\end{array} \qquad \begin{array}{ccccc}
& & 1 \\
& 0 && 0 \\
m-1 && 10m && m-1 \\
& 0 && 0 \\
&& 1
\end{array} \]                   
The surfaces in Theorem~\ref{thm1} have $m=2$ and those in
Theorem~\ref{thm2} have $m=3$.  To tie in with \cite[Theorem 4]{KS},
we note that $p_g = h^{2,0} = m- 1$.  By the Lefschetz theorem on
$(1,1)$-classes we have $\rho \le h^{1,1} = 10 m$.  This already
determines $\rho$ in all cases with $N \le 8$.  It remains for us to
improve this upper bound by $1$ in the cases
$(N,\eps) = (9,1),(12,1),(9,2)$, and to improve it by $2$ in the cases
$(N,\eps) = (10,1),(10,3),(11,1)$.

The main tool we wish to use (see \cite[Proposition 6.2]{vL-heron}) is
that there is an injective map
$\NS(\overline{X}) \to \NS(\overline{X}_p)$ that preserves the
intersection pairing.

Let $f_p(x)$ be the characteristic polynomial of Frobenius acting on
$H^2_{{\text{\'et}}}(\overline{X}_p,\Q_\ell(1))$, normalised so that
$f_p(0)=1$.  This is a polynomial of degree $b_{2} = 12m - 2$,
independent of the choice of prime $\ell \not=p$. By the Weil
conjectures it satisfies the functional equation
$f_p(x) = \pm x^{b_2} f_p(1/x)$.  The polynomials $f_p(x)$ may be
computed from the numbers $n_r = | X_p(\F_{p^r}) |$ using the
Lefschetz trace formula.  See for example \cite[Section 3]{vL1}, where
$f_p$ is denoted $\widetilde{f}_p$.  We used both the functional
equation and the Magma function {\tt FrobeniusActionOnTrivialLattice}
to limit how many $n_r$ we had to compute.
The polynomials $f_p(x)$ for two carefully chosen primes of good reduction
are recorded in Table~\ref{tab:frob}.
\begin{table}[ht]
\caption{} 
\label{tab:frob}
$\begin{array}{ll}
(N,\eps) & \multicolumn{1}{c}{\text{ Characteristic polynomial of Frobenius }}  \\ \hline
(9,1) & f_5(x) = (x - 1)^{16} (x + 1)^2 (x^2 + x + 1) (x^2 + \frac{7}{5} x + 1) \\
      & f_7(x) = (x - 1)^{18} (x + 1)^2 (x^2 + \frac{10}{7} x + 1) \\
(12,1) & f_5(x) = (x - 1)^{16} (x + 1)^4 (x^2 + \frac{6}{5} x + 1) \\
       & f_{11}(x) = (x - 1)^{12} (x + 1)^8 (x^2 + \frac{6}{11} x + 1) \\
(9,2) & f_7(x) = (x - 1)^{24} (x + 1)^2 (x^2 + x + 1)^2 (x^2 + \frac{10}{7} x + 1) (x^2 + \frac{13}{7} x + 1) \\
      & f_{13}(x) = (x - 1)^{24} (x^2 + x + 1)^3 (x^2 + \frac{1}{13} x + 1) (x^2 + \frac{25}{13} x + 1) 
\\
(10,1) & f_7(x) = (x - 1)^{24} (x + 1)^2 (x^2 + x + 1)^2 (x^2 + \frac{10}{7} x + 1)^2 \\
       & f_{17}(x) = -(x - 1)^{25} (x + 1)^5 (x^2 - \frac{2}{17} x + 1) (x^2 + \frac{25}{17} x + 1) \\
(10,3) & f_{31}(x) =  (x - 1)^{24} (x + 1)^2 (x^2 + x + 1)^2 (x^2 + \frac{46}{31} x + 1) (x^2 + \frac{58}{31} x + 1)
\\
       & f_{37}(x) = (x - 1)^{28} (x + 1)^2 (x^2 + \frac{70}{37} x + 1)^2 \\
(11,1) & f_{23}(x) =  (x - 1)^{28} (x + 1)^2 (x^2 + \frac{42}{23} x + 1) (x^2 + \frac{45}{23} x + 1) \\
       & f_{53}(x) = (x - 1)^{28}  (x^2 + x + 1) (x^2 + \frac{25}{53} x + 1) (x^2 + \frac{70}{53} x + 1) 
\end{array}$ 
\end{table}

Let $\Delta_p \in \Q^\times/(\Q^\times)^2$ be the absolute value of
the determinant of the intersection pairing on $\NS(\overline{X}_p)$.
It may be computed using either of the following two lemmas.

\begin{Lemma} 
\label{lem:Kl}
Write $f_p(x) = g_p(x) h_p(x)$ where every root of $g_p$ is a root of
unity, and no root of $h_p$ is a root of unity. Then
$\rho_p \le \deg g_p$ and in the case of equality we have
$\Delta_p \equiv h_p(1) h_p(-1) \mod{(\Q^\times)^2}$.
\end{Lemma}
\begin{proof}
  The first part is described for example in \cite{vL1}. The Tate
  conjecture predicts that this inequality is always an equality, and
  this has been proved in many cases.  See \cite[Section 17.3]{H} for
  the history of this problem and further references.  The formula for
  $\Delta_p$ is a small refinement of a result of Kloosterman, that in
  turn depends on known cases of the Artin-Tate conjecture.

  Let $F_k(x) = \prod (1 - p^k \alpha_i^k x)$ where
  $f_p(x) = \prod(1- \alpha_ix)$.  The result of Kloosterman
  \cite[Proposition 4.7]{Kl}, is that if $k$ is an even integer with
  $\alpha_i^k=1$ for all $\alpha_i$ which are roots of unity, then
\begin{equation}
\label{eqn:K}
\Delta_p = \lim_{s \to 1} \frac{F_k(p^{-ks})}{(1-p^{k(1-s)})^\rho}.
\end{equation}
Let $H_k(x) = \prod (1- p^k \beta_i^k x)$ where $h_p(x) = \prod(1- \beta_ix)$.
Then $F_k(x) = (1-p^k x)^\rho H_k(x)$, and~\eqref{eqn:K} becomes
\[ \Delta_p = H_k(p^{-k}) = \prod_i (\beta_i^k-1) = h_p(1)h_p(-1)
\prod_{{\substack{d \mid k\\d > 2}}} \prod_i \Phi_d(\beta_i) \]
where $\Phi_d$ is the $d$th cyclotomic polynomial.  For $d > 2$ we
claim that $\prod_i \Phi_d(\beta_i)$ is a rational square.  By the
functional equation we have
\[\beta_{1}, \ldots, \beta_{2m} = \gamma_1, \ldots, \gamma_m,
\gamma_1^{-1}, \ldots , \gamma_m^{-1}.\]
Since $d>2$ we have $\Phi_d(x) = x^{\phi(d)} \Phi_d(1/x)$ where
$\phi(d)$ is even, say $\phi(d) = 2n$. Therefore
$\gamma_i^{-n} \Phi_d(\gamma_i) = \gamma_i^{n}
\Phi_d(\gamma_i^{-1})$.
It follows that $\prod_{i=1}^m \gamma_i^{-n} \Phi_d(\gamma_i) \in \Q$
and
\[ \prod_{i=1}^{2m} \Phi_d(\beta_i) = \prod_{i=1}^{2m} \beta_i^{-n}
\Phi_d(\beta_i) = \bigg(\prod_{i=1}^{m} \gamma_i^{-n}
\Phi_d(\gamma_i)\bigg)^2. \qedhere \]
\end{proof}

\begin{Lemma} 
\label{lem:BSD}
Suppose that $P_1, \ldots P_r$ generate a finite index subgroup of
$E_p(\Fbar_p(T))$.  Then we have
$\Delta_p \equiv (\prod_t c_t) \Reg(P_1, \ldots, P_r)
\mod{(\Q^\times)^2}$
where the product is over $t \in \PP^1(\Fbar_p)$ and $c_t$ is the
number of irreducible components of multiplicity one in the fibre of
$\overline{X}_p$ above $t$.
\end{Lemma}
\begin{proof}
See \cite[Theorem 8.7 and (7.8)]{Shioda}.
\end{proof}

In the calculations below, we sometimes needed to find explicit
generators for $E_p(\F_p(T))$. These were found by searching on
$2$-coverings, computed using $2$-descent in Magma, as implemented in the
function field case by S. Donnelly.

If $\rho = \rho_p = \rho_q$ for distinct primes $p$ and $q$, then we
have $\Delta_p \equiv \Delta_q \mod{(\Q^\times)^2}$.  As observed by
van Luijk \cite{vL1}, this can sometimes be used to improve our upper
bound on $\rho$ by $1$. This is particularly useful since (assuming
the Tate conjecture) $\rho_p$ is always even. Indeed
$\rho_p = \deg g_p = b_2 - \deg h_p$, and $\deg h_p$ is even by the
functional equation.  See \cite{Kl} and \cite{BN} for further
examples.

\medskip

\noindent
{\bf Case ({\em N},\boldmath$\eps$) = (9,1).}
We already know that $\rho = 19$ or $20$. Since the Tate conjecture
has been proved for elliptic K3-surfaces, equality holds in
Lemma~\ref{lem:Kl}.  By Lemma~\ref{lem:Kl} we compute
$\Delta_5 = 3 \cdot 17$ and $\Delta_7 = 2 \cdot 3$.  Since these are
different, it follows by the method of van Luijk that $\rho = 19$.

\medskip

\noindent
{\bf Case ({\em N},\boldmath$\eps$) = (12,1).}
This is identical to the previous example, except that now
$\Delta_5=1$ and $\Delta_{11}=7$.

\medskip

\noindent
{\bf Case ({\em N},\boldmath$\eps$) = (9,2).}
We already know that $\rho = 29$ or $30$. Let $p=7$ or $13$. By
Lemma~\ref{lem:Kl} and~\eqref{STp} 
we have $\rank \,\, E_p(\Fbar_p(T)) \le 3$. We prove equality by
exhibiting three independent points of infinite order in
$E_p(\F_p(T))$.  In addition to the reductions of the two points in
Table~\ref{tabMW}, we have when $p=7$ the point
\[( 5 T^5 + 6 T^4 + 4 T^2 + 6 T, T^7 + 3 T^6 + 6 T^3 + 2 T^2 + 2 T), \]
and when $p=13$ the point
\[ ( 4 T^6 + 8 T^5 + 3 T^4 + 7 T^3 + 5 T^2 + 10 T + 2, 10 T^9 + 4 T^8
+ 5 T^6 + 5 T^5 + 12 T^3 + 4 T^2 + 12). \]
Using either Lemma~\ref{lem:Kl} or Lemma~\ref{lem:BSD} we find that
$\Delta_7 = 2$ and $\Delta_{13} = 17$.  Since these are different, it
follows that $\rho = 29$.

\medskip

In the cases $(N,\eps) = (10,1),(10,3),(11,1)$ we aim to show that
$\rho = 28$.  We were unable\footnote{There is presumably a 
systematic reason for this, similar to that described in \cite{Charles}.}
to find a prime $p$ with
$\rho_p = 28$, despite computing the polynomials $f_p(x)$ for all
$p < 150$.  This prompted us to try a variant of van Luijk's method,
were we use the intersection pairing to improve our upper bound for
$\rho$ by $2$.

\medskip

\noindent
{\bf Case ({\em N},\boldmath$\eps$) = (10,1).}
We already know that $\rho = 28$, $29$ or $30$.  In addition to the
point $P_1 = (0,0)$ in Table~\ref{tabMW} we have when $p=7$ the points
\begin{align*}
Q_1 &= (6 T^6 + 6 T^4 + 4 T^3 + 5 T^2,4 T^9 + 6 T^8 + 6 T^7 + T^6 + T^5 + 3 T^4), \\
Q_2 &= (T^6 + 5 T^5 + 6 T^4 + 4 T^3 + 5 T^2,2 T^9 + 6 T^8 + 2 T^7 + T^6 + 3 T^4),
\end{align*}
and when $p=17$ the points
\begin{align*}
R_1 &= (16 T^6 + 13 T^5 + 6 T^4 + 4 T^3 + 12 T^2,
       4 T^9 + 2 T^8 + 5 T^7 + 8 T^5 + 15 T^4), \\
R_2 &= ((6 T^8 + 8 T^7 + 2 T^6 + 5 T^5 + 8 T^4 + 4 T^3 + T^2)/(T + 6)^2, \,\, \ldots \,\,).
\end{align*}
Using either Lemma~\ref{lem:Kl} or Lemma~\ref{lem:BSD} we find that
$\Delta_7 = 1$ and $\Delta_{17} = 2 \cdot 59$. Since these are
different, it follows that $\rho \le 29$.

Reducing mod $7$ or $17$ does not change the Kodaira symbols of the
singular fibres. So by Lemma~\ref{lem:BSD} it will be enough for us to
work with the height pairing on the Mordell-Weil group, rather than
the intersection pairing on the full N\'eron-Severi group. We compute
\begin{align*}
\Reg(P_1,uQ_1 + vQ_2) & = \tfrac{2}{75} ( 7 u^2 - 12 uv + 18 v^2), \\
\Reg(P_1,xR_1 + yR_2) & = \tfrac{1}{450} ( 139 x^2 + 76 xy + 316 y^2).
\end{align*}
If $\rho = 29$ then the equation
$\tfrac{2}{75} ( 7 u^2 - 12 u v + 18 v^2) = \tfrac{1}{450} ( 139 x^2 +
76 xy + 316 y^2)$
has a solution in rational numbers $u,v,x,y$ not all zero. However
this quadratic form of rank $4$ is not locally soluble over the
$3$-adics.  Therefore $\rho = 28$.

\medskip

\noindent
{\bf Case ({\em N},\boldmath$\eps$) = (10,3).}
We already know that $\rho = 28$, $29$ or $30$.  Let $p = 31$ or
$37$. Since $p \equiv 1 \pmod{3}$ the reductions of the points in
Tables~\ref{tabMW} and~\ref{tabMWgeom} give us points
$P_1,P_2,P_3,P_4 \in E_p(\F_p(T))$. In addition when $p=31$ we have
\begin{align*}
Q_1 &= (20 T^4 + 13 T^3 + 30 T^2 + 6 T,5 T^5 + 4 T^4 + 29 T^3 + 4 T^2 + 5 T), \\
Q_2 &= (7 T^6 + 12 T^5 + 9 T^4 + 19 T^2 + 13 T + 4)/(T+29)^2, \,\, \ldots \,\, ),
\end{align*}
and when $p=37$ we have 
\begin{align*}
R_1 &= (36 T^4 + 11 T^3 + 4 T^2,26 T^5 + 34 T^4 + 2 T^3 + 15 T^2), \\
R_2 &= (6 T^4 + 5 T^3 + T^2 + 26 T + 32,
       29 T^5 + 35 T^4 + 2 T^3 + 15 T^2 + 19 T + 10).
\end{align*}
Using either Lemma~\ref{lem:Kl} or Lemma~\ref{lem:BSD} we find that
$\Delta_{31} = 2 \cdot 5$ and $\Delta_{37} = 1$. Therefore
$\rho \le 29$.

As in the previous example, reducing mod $31$ or $37$ does not change
the singular fibres.  We compute
\begin{align*}
\Reg(P_1,P_2,P_3,P_4,uQ_1 + vQ_2) & = \tfrac{5}{96} ( 25 u^2 - 4 uv + 52 v^2), \\
\Reg(P_1,P_2,P_3,P_4,xR_1 + yR_2) & = \tfrac{1}{8} ( 5 x^2 + 8 y^2).
\end{align*}
If $\rho = 29$ then the equation
$\tfrac{5}{96} ( 25 u^2 - 4 uv + 52 v^2) = \tfrac{1}{8} ( 5 x^2 + 8
y^2)$
has a solution in rational numbers $u,v,x,y$ not all zero. However
this quadratic form is not locally soluble over the $3$-adics.
Therefore $\rho = 28$.

\medskip

\noindent
{\bf Case ({\em N},\boldmath$\eps$) = (11,1).}
We already know that $\rho = 28$, $29$ or $30$.
Let $P_1$ and $P_2$ be the reductions mod $p$ of the points in Table~\ref{tabMW}.
In addition, when $p=23$ we have
\begin{align*}
Q_1 &= (16 T^2 + 5 T + 5,21 T^3 + 15 T^2 + 3 T + 18), \\
Q_2 &= (18 T^6 + 5 T^5 + 5 T^4 + 22 T^3 + 9 T^2)/(T+16)^2, \,\, \ldots \,\, ),
\end{align*}
and when $p=53$ we have
\begin{align*}
R_1 &= (28 T^5 + T^4 + 23 T^3 + 40 T^2 + 15 T, \,\, \ldots \,\, ), \\
R_2 &= (49 T^6 + 44 T^5 + 38 T^4)/(T^2 + 42 T + 5)^2,  \,\, \ldots \,\, ).
\end{align*}
Using either Lemma~\ref{lem:Kl} or Lemma~\ref{lem:BSD} we find that 
$\Delta_{23}=  2 \cdot 7 \cdot 11 \cdot 13$ and $\Delta_{53} = 11 \cdot 131$.
Therefore $\rho \le 29$. 

Again, reducing mod $23$ or $53$ does not change the singular fibres.
We compute
\begin{align*}
\Reg(P_1,P_2,uQ_1 + vQ_2) & = \tfrac{11}{480} ( 57 u^2 - 46 uv + 137 v^2), \\
\Reg(P_1,P_2,xR_1 + yR_2) & = \tfrac{1}{240} ( 541 x^2 - 228xy + 1196 y^2).
\end{align*}
If $\rho = 29$ then the equation 
$\tfrac{11}{480} ( 57 u^2 - 46 uv + 137 v^2) = \tfrac{1}{240} ( 541 x^2 - 228xy + 1196 y^2)$
has a solution in rational numbers $u,v,x,y$ not all zero. However
this quadratic form is not locally soluble over the $11$-adics.
Therefore $\rho = 28$.

\end{document}